\documentclass[12pt,twoside,a4paper,reqno]{amsart}
\usepackage{amsmath,amsthm,amssymb,amscd,amsfonts}
\usepackage{enumerate}
\makeatletter
\@namedef{subjclassname@2010}{%
  \textup{2010} Mathematics Subject Classification}
\makeatother
\theoremstyle{definition}

\newtheorem{theo}{Theorem}[section]

\newtheorem{lemm}[theo]{Lemma}
\newtheorem{prop}[theo]{Proposition}
\theoremstyle{definition}

\theoremstyle{definition}

\newtheorem{rem}[theo]{Remark}
\numberwithin{equation}{section}

\title{Amenability-like properties of $C(X, A)$}

\author{S. Ghoraishi}
\address{ Department of Mathematics, Central Tehran Branch, Islamic Azad University, Tehran, Iran.
 e-mail: {\tt esm.ghoraishimadineh.sci@iauctb.ac.ir }}

\author{A. Mahmoodi}
\address{  Department of Mathematics, Central Tehran Branch, Islamic Azad University, Tehran, Iran.
 e-mail: {\tt a\_mahmoodi@iauctb.ac.ir}}

\author{A. R. Medghalchi}
\address{  Faculty of Mathematical Sciences and Computer, Kharazmi University, 50 Taleghani Avenue, Tehran, Iran.
 e-mail: {\tt a\_medghalchi@khu.ac.ir}}

\begin{document}

\pagestyle{headings}

\begin{abstract}
Let $A$ be a Banach algebra and $X$ be a compact Hausdorff space.
For given homomorphisms $ \sigma \in Hom(A)$ and $\tau \in Hom(C(X,
A))$, we introduce homomorphisms $\tilde{\sigma}\in Hom(C(X, A)) $
and $\tilde{\tau}_x \in Hom(A)$, where $x \in X$. We then study both
$\tilde{\sigma}$-(weak) amenability of $C(X, A)$, and
$\tilde{\tau}_x$-(weak) amenability of $A$.
\end{abstract}

\maketitle Keywords:

   $\sigma$-amenability,
    $\sigma$-weak amenability, $\sigma$-bounded approximate
    diagonal.

MSC 2010: Primary: 46M10; Secondary: 46M18, 46H20.
\section{Introduction}
Let $X$ be a compact Hausdorff space and let $A$ be a Banach
algebra. It is known that $\mathit{C(X,A)}$, the set of all
$A$-valued continuous functions on $X$, is a Banach algebra with
pointwise algebraic operations and the uniform norm $||f||_\infty :=
\sup_{x \in X} ||f(x)||$, $f \in C(X,A)$ \cite{K}. In the nice
papers \cite{R.Y,Y}, Ghamarshoushtari and Zhang studied amenability
and weak amenability of $\mathit{C(X,A)}$. They showed that
$\mathit{C(X,A)}$ is amenable if and only if $A$ is amenable.
Further, if $A$ is commutative, they proved that $C(X,A)$ is weakly
amenable if and only if $A$ is weakly amenable.

Let $A$ be a Banach algebra. We denote by $ Hom(A)$ the space of all
continuous homomorphisms from $A$ into $A$. Let $A$ be a Banach
algebra, $E$ be a Banach $A$-bimodule and let $\sigma \in Hom(A)$.
 A bounded linear map $D: A\longrightarrow E$ is $ \sigma$-\textit{derivation} if $ D(ab)=D(a)\cdot \sigma(b)+\sigma(a)\cdot D(b)$, for $a, b\in A$.
 A $ \sigma$-derivation $D:A\longrightarrow E$ is $ \sigma$-\textit{inner derivation} if there exists $x\in E$ such that
 $D(a)= \sigma(a)\cdot x-x \cdot \sigma(a)$, for all $a\in A$. A Banach algebra $A$ is $ \sigma$-\textit{amenable}
 if for every Banach $A$-bimodule $E$, every $ \sigma$-derivation $D: A\longrightarrow
 E^*$ is $ \sigma$-inner. Especially, a Banach algebra $A$ is $ \sigma$-\textit{weakly amenable} if every $ \sigma$-derivation $D: A\longrightarrow
 A^*$ is $ \sigma$-inner. It is known that the projective tensor
 product $A\hat{\otimes} A$ is a Banach $A$-bimodule in a natural
 way. A bounded net $(u_ {\alpha})\subset A\hat{\otimes}A$
  is a $\sigma$-\textit{bounded approximate diagonal} for $A$ if
 $$\sigma(a)\cdot u_{\alpha} -u_{\alpha}\cdot\sigma(a) \longrightarrow 0  ,\ \ \ \text{and} \ \ \  \pi(u_{\alpha}) \sigma(a) \longrightarrow \sigma(a)
 \ \ \ ( a \in A) , $$ where $\pi : A\hat{\otimes} A \longrightarrow A$ is the product map defined by $\pi(a\hat{\otimes} b) = ab $.

Before preceding further, we set up our notations. Let $X$ be a
compact Hausdorff space and let $A$ be a Banach algebra. For each $x
\in X$, we consider the continuous epimorphism  $ T_x:
C(X,A)\longrightarrow A $ defined by $ T_x(f) := f(x)$, for all $ f
\in C(X, A)$. For every $ a \in A$, we define $ 1_a \in C(X, A)$ by
the formula $1_a (x) := a$, for all $x \in X$. We notice that, every
homomorphism $\sigma \in Hom(A)$ induces a homomorphism $
\tilde{\sigma} \in Hom(C(X, A))$ defined by $\tilde{\sigma}(f)
:=\sigma f$, $f\in C(X, A)$. Conversely, for every homomorphism
$\tau \in Hom (C(X, A))$ and  $x \in X$, we introduce the map $
\tilde{\tau}_x: A\longrightarrow A $ through $ \tilde{\tau}_x(a) :=
T_x(\tau(1_{a}))$, $a \in A$. Since $1_{ab} = 1_a 1_b$,
$\tilde{\tau}_x \in Hom(A)$. It is readily seen that $(T_x
\tau)(1_a) = (\tilde{\tau}_x T_x) (1_a) $, for each $ a \in A$.
Next, for $f \in C(X)$ and $ a \in A$ we define $f a \in C(X, A)$
via $ f a(x) = f(x) a$ for each $x \in X$. Throughout the paper, we
keep the above definitions and notations.

In this paper, we deal with amenability-like properties of the
Banach algebra $C(X, A)$. Suppose that $\sigma$ and $\tau$ are
homomorphisms on $A$ and $C(X, A)$, respectively. For a given
$\sigma$-bounded approximate diagonal for $A$, we construct a $
\tilde{\sigma}$-bounded approximate diagonal for $C(X, A)$ (Theorem
2.1).  We show that under some certain conditions, $\tau$-(weak)
amenability of $C(X, A)$ implies $ \tilde{\tau}_x $-(weak)
amenability of $A$ (Theorems 2.4 and 3.4). For a commutative Banach
algebra $A$, we prove that $\sigma$-weak amenability of $A$ yields $
\tilde{\sigma}$-weak amenability of $C(X, A)$ in the presence of a
bounded approximate identity for $A$ (Theorem 3.2). Finally, we show
that Theorem 3.2 is still true without the existence of a bounded
approximate identity for $A$ (Theorem 3.9).

\section{$\sigma$-amenability}
Suppose that $A$ is a Banach algebra and $X$ is a compact Hausdorff
space. For $ u = \sum_{i} u_{i} \otimes v_{i} \in C(X) \hat{\otimes}
C(X)$ and $ \alpha =  \sum _{j} \alpha_{j} \otimes \beta_{j}\in
A\hat{\otimes} A$, we consider
$$\Gamma(u, \alpha)=  \sum _{i, j} u_{i}\alpha_{j} \otimes
v_{i}\beta{j} \in C(X, A) \hat{\otimes}C(X, A)$$ so that  $\Vert
\Gamma(u, \alpha ) \Vert \leq \Vert u \Vert \  \Vert  \alpha \Vert$.

\begin{theo}\label{2.1}
Let $X$ be a compact Hausdorff space, $A$ be a Banach algebra and
let $\sigma \in Hom(A)$. If $A$ has a $ \sigma $-bounded approximate
diagonal, then $ C(X, A) $ has a $\tilde{\sigma}$-bounded
approximate diagonal.

\end{theo}
\begin{proof} We follow the standard argument in \cite{R.Y}. Let $(\alpha_{v}) \in A \hat{\otimes} A$ be a $\sigma$-bounded
approximate diagonal
  for $ A$ such that $\Vert \alpha_{v} \Vert_{p}\leq M$ for all $v$.
  We claim that for any $\varepsilon > 0$
   and any finite set $ F \subset C(X, A)$, there is $ U= U_{(F, \varepsilon)}\in C(X, A) \hat{\otimes}C(X, A)$ with $\Vert U \Vert \leq 2Mc$ such that
 $$  \Vert \tilde{\sigma}(a) \ . \ U - U \ . \ \tilde{\sigma}(a)  \Vert< \varepsilon \ \ \ \text{and} \ \ \ \Vert  \pi(U)
  \tilde{\sigma}(a)-\tilde{\sigma}(a)\Vert < \varepsilon
  \ \ \ (a \in F) $$ where $c > 0$
  is the constant asserted in [2, Corollary 1.2]. Given $\varepsilon > 0$ and a finite set $ F \subset  C(X, A)$.
  We first assume that each $a \in F$ is of the form of a finite sum
  $a =  \sum_{k} f_{k} a_{k} $, with $ f_{k}\in C(X) $ and $a_{k} \in
  A$. It is easy to check that $\tilde{\sigma}(a) =   \sum_{k} f_{k} \sigma(a_{k}) $. We denote by $ F_{A} $ $( \subset A)$ the finite set of all
  elements $a_{k}$ associated to $a$ for all $a \in F$, and by $F_{C}$ $(\subset
  C(X))$ the finite set of all functions $f_{k}$ associated to $a$ for all $a \in
  F$.  Let $N> 0$ be an integer that is greater than the number of the terms of  $a =  \sum_{k} f_{k} a_{k} $
  for all $a \in F$, and set $ L := \max \{   \Vert \sigma(b) \Vert \ , \ \Vert f \Vert \ : \ b \in F_{A} ,  f\in
  F_{C}\}$. By the assumption,  there is $\alpha \in (\alpha_{v})$ such that
  $$ \Vert \sigma(b) \ . \ \alpha-\alpha \ . \ \sigma(b) \Vert < \dfrac{\varepsilon}{4cNL}, \ \ \  \Vert \pi (\alpha) \ . \ \sigma(b) -
 \sigma(b) \Vert < \dfrac{\varepsilon}{NL} \ \ \ (b \in F_{A}) .$$  On the other hand,
  by the same argument as in the proof of [2, Theorem 2.1], we obtain an element $u \in C(X) \otimes C(X) $ with
   $\pi(u) =1$ and $\Vert  u \Vert \leq 2c$ for which
 $$\Vert f \ . \ u - u \ . \ f \Vert  <
  \dfrac{\varepsilon}{2 \Vert \alpha \Vert LN}        \ \ \ (f \in
  F_{C}).$$ Putting $U:=\Gamma(u, \alpha) $, for $a =
 \sum_{k} f_{k} a_{k}$ we see that
\begin{align*}
\Vert \tilde{\sigma}(a) \ . \ U - U \ . \ \tilde{\sigma}(a)\Vert &=
\Vert \sum _{k}(\Gamma(f_{k}u, \sigma(a_{k})\alpha)-\Gamma(uf_{k},
\alpha \sigma(a_{k})) \Vert \\&= \Vert  \sum_{k}\Gamma(f_{k}u,
\sigma(a_{k}) \alpha - \alpha \sigma(a_{k})) + \Gamma(f_{k}u-uf_{k},
\alpha \sigma(a_{k}) )\Vert \\& \leq  \sum_{k}(L \Vert u\Vert \
\Vert \sigma(a_{k})\alpha -\alpha \sigma(a_{k})\Vert +L \Vert\alpha
\Vert \  \Vert f_{k}u-uf_{k}\Vert) \\& \leq NL(2c
 \dfrac{\varepsilon}{4cNL}+ \Vert \alpha\Vert
  \dfrac{\varepsilon}{2\Vert  \alpha
   \Vert_{P} LN})=  \varepsilon
\end{align*}
and
\begin{align*}
\Vert  \pi(U) \tilde{\sigma}(a)-\tilde{\sigma}(a)\Vert &= \Vert
\pi(u) \pi(\alpha) \tilde{\sigma}(a)-\tilde{\sigma}(a)\Vert = \Vert
\sum_{k} f_{k}(\pi(\alpha) \sigma(a_{k})-\sigma(a_{k})) \Vert \\&
\leq L \sum _{k} \Vert \pi(\alpha) \sigma(a_{k})-\sigma(a_{k})\Vert
< NL  \dfrac{\varepsilon}{NL} = \varepsilon.
\end{align*}

Now, we assume that $ F \subset C(X, A)$ is an arbitrary finite set.
From the proof of [2, Theorem 2.1], we know that for each $ a \in F
$ there exists an element $a_{\varepsilon} =  \sum_{k} f_{k} a_{k} $
where the right side of $ a_{\varepsilon}$ is a finite sum, $
f_{k}\in C(X)$ and $ a_{k}\in A$ such that $ \Vert \tilde
{\sigma}(a)-\tilde{\sigma}(a_{\varepsilon}) \Vert <  \min \big
\lbrace \dfrac{\varepsilon}{4}, \dfrac{\varepsilon}{8Mc} \big
\rbrace $. Applying the above argument for the finite set $
F_{\varepsilon} := \lbrace a_{\varepsilon}   : a \in F \rbrace $, we
get $ U \in C(X, A) \hat{\otimes} C(X, A)$ such that $ \Vert U \Vert
\leq 2cM $ and $$\Vert  \tilde{\sigma}(a_{\varepsilon}) \ . \ U - U
\ . \ \tilde{\sigma}(a_{\varepsilon})\Vert < \dfrac{\varepsilon}{2}
\ \ \ \text{and} \ \ \ \Vert  \pi(U) \tilde{\sigma}(a_{\varepsilon})
- \tilde{\sigma}(a_{\varepsilon})\Vert < \dfrac{\varepsilon}{2}  \ \
\ ( a_{\varepsilon} \in F_{\varepsilon}).
$$ Therefore $$\Vert  \tilde{\sigma}(a) \ . \ U - U \ . \
\tilde{\sigma}(a)\Vert < \varepsilon \ \ \ \text{and} \ \ \ \Vert
\pi(U) \tilde{\sigma}(a) - \tilde{\sigma}(a)\Vert < \varepsilon  \ \
\ ( a \in F)
$$ so that the claim is proved. Finally, the net $(U_{F, \varepsilon})$ with the natural
    partial order $(F_{1}, \varepsilon_{1}) < (F_{2}, \varepsilon_{2})$ if and only if
    $ F_{1}\subset F_{2}$ and $\varepsilon_{1} \geq \varepsilon_{2}
    $ is the desired  $\tilde{\sigma}$-approximate diagonal for $C(X,
    A)$.
\end{proof}
\begin{rem} \label{2.2}
To our knowledge, we do not know whether or not the existence of
$\sigma$-bounded approximate diagonal is equivalent to
$\sigma$-amenability. Hence, we can not prove or disprove if
$\sigma$-amenability of $A$ implies $\tilde{\sigma}$-amenability of
$C(X,   A)$.
\end{rem}
\begin{prop}
Suppose that $\sigma \in Hom(A)$ and $\tau \in Hom(B)$, where $A$
and $B$ are Banach algebras. Suppose that $ \phi : A \longrightarrow
B$ is a continuous homomorphism with a dense range and $ \tau \phi =
\phi \sigma$. If $A$ is $\sigma$-amenable, then $B$ is
$\tau$-amenable.
\end{prop}
\begin{proof}
We may either prove it or else look at \cite [Proposition 3.3]{M.
A}.
\end{proof}
\begin{theo}
 Let $X$ be a compact Hausdorff space, let $A$ be a Banach algebra
and let $\tau \in Hom(C(X, A))$ such that  $T_{x_0} \tau =
\tilde{\tau}_{x_0} T_{x_0}$, for some $x_0 \in X$. If $C(X, A)$ is
$\tau$-amenable, then $A$ is $\tilde{\tau}_{x_0}$-amenable.
\end{theo}
\begin{proof}
We have already seen that the map $T_{x_0}$ is surjective, so this
is immediate by Proposition 2.3.
\end{proof}
We note that homomorphisms $\tau \in Hom(C(X, A))$ satisfying the
condition of Theorem 2.4 exist in abundance. Take an arbitrary
homomorphism $\eta \in Hom(C(X, A))$ and define the map $ \tau :
C(X, A) \longrightarrow C(X, A)$ by $$ \tau(f)(x) := \tau
(1_{f(x)})(x) := \eta (1_{f(x)})(x) \ \ \  ( f \in C(X, A), x \in X)
\ .
$$ Then, we may check that $\tau \in Hom(C(X, A))$ and $T_x \tau =
\tilde{\tau}_x T_x$ for all $x \in X$.

\section{$\sigma$-weak amenability}
For a Banach algebra $A$ and a homomorphism $\sigma$ on $A$, we
write $Z^{1}_{\sigma}(A,E)$ for the set of all $\sigma$-derivations
from $A$ into a Banach $A$-bimodule $E$.

Let $A$ be a $\sigma$-weakly amenable Banach algebra, and $ \sigma
\in Hom(A)$ such that $ \sigma (A)$ is a dense subset of $A$. A more
or less verbatim of the classic argument, shows that $
\overline{A^{2}}=A$ (see [6, Theorem 6] for details).

\begin{prop} \label{3.4}
Let $A$ be a commutative Banach algebra and let $ \sigma \in Hom(A)$
with a dense range. If $A$ is $\sigma$-weakly amenable, then
$Z^{1}_{\sigma}(A,E)={0}$ for each Banach $A$-module $E$.
\end{prop}
\begin{proof}
Assume towards a contradiction that there is a nonzero element $D
\in Z^{1}_{\sigma}(A,E)$. Since $ \overline{A^{2}}=A$, there exists
$a_{0}\in A$ with $D(a_{0}^{2})\neq 0$. Hence $\sigma(a_{0}) \cdot
Da_{0}\neq 0$ and so there exists $\lambda \in E^{*}$ with $\langle
\sigma(a_{0}) \cdot Da_{0}, \lambda \rangle =1$. Consider
$R_{\lambda}\in {}_{A}{B(E,A^{*})}$ such that $\langle a,
R_{\lambda}x\rangle = \langle a \cdot x, \lambda\rangle$, for $ a
\in A$ and $ x \in E$ \cite[Proposition 2.6.6]{H. G}. It is not hard
to see that $R_{\lambda} \circ D \in Z^{1}_{\sigma}(A, A^{*})$. Then
we have $$ \langle \sigma(a_{0}), (R_{\lambda} \circ
D)(a_{0})\rangle = \langle \sigma(a_{0})\cdot Da_{0}, \lambda
\rangle =1  $$ so that $R_{\lambda} \circ D \neq0$, a contradiction
of the fact that $A$ is $\sigma$-weakly amenable.
\end{proof}

\begin{prop}\label{3.2}
Let $X$ be a compact Hausdorff space and let $A$ be a commutative
Banach algebra with a bounded net $(e_{v})$ for which $
(\sigma(e_{v}))$ is a bounded approximate identity for $A$, where $
\sigma$ belongs to $ Hom(A)$ with a dense range. If $A$ is
$\sigma$-weakly amenable, then $C(X, A)$ is $\tilde{\sigma}$-weakly
amenable.
\end{prop}
\begin{proof}
Using the map $ a \longmapsto 1_a$, we may consider $A$ as a closed
subalgebra of the commutative Banach algebra $C(X,A)$. Suppose that
$ D : C(X, A) \longrightarrow C(X, A)^{*}$ is a
$\tilde{\sigma}$-derivation. Notice that $C(X,A)$ is naturally a
commutative $A$-bimodule with actions $a \ . \ f (x) = f \ . \ a (x)
:= a f(x)$, ($a \in A, f \in C(X,A), x \in X$). We also note that
$\tilde{\sigma} = \sigma $ on $A$. Therefore $ D\mid _{A} : A
\longrightarrow C(X, A)^{*} $ is a $\sigma$-derivation and then, by
Proposition 3.1, $D\mid _{A} = 0 $. On the other hand, $
(\sigma(e_{v}))$ is also a bounded approximate identity for $C(X,
A)$ and $ wk^{*}$-$\lim D(\sigma(e_{v})) = 0$. An argument similar
to that in the proof of [2, Proposition 4.1] shows that $
wk^{*}$-$\lim  D(f \sigma(e_{v}))$ exists for each $ f \in C(X)$. So
we may define $\tilde{D} : C(X) \longrightarrow C(X,A)^{*}$ via $
\tilde{D}(f) = wk^{*}-\lim_{v} D(f\sigma(e_{v}))$. Clearly $C(X,A)$
is a commutative $ C(X)$-bimodule. Then
\begin{align*}
D(fg\sigma(e_{v})) &= wk^{*}-\lim D(f g \sigma(e_{\mu})
\sigma(e_{v}))) = wk^{*}-\lim D(f\sigma(e_{\mu})g\sigma(e_{v})))
\\&= wk^{*}-\lim D(f\sigma(e_{\mu})) \ . \ \tilde{\sigma}(g\sigma(e_{v})
) + \tilde{\sigma} (f\sigma(e_{\mu})) \ . \ D(g\sigma(e_{v})) \\&=
\tilde{D}(f) \ . \ \tilde{\sigma}(g\sigma(e_{v}) ) + f \ . \
D(g\sigma(e_{v}))
\end{align*}
for $f,g \in C(X)$. Taking $wk^{*}$-limit in $v$, we get $$
\tilde{D}(fg) =  \tilde{D}(f) \ . \ g + f \ . \ \tilde{D}(g) \ \ \
(f,g \in C(X) ) .$$  Hence $\tilde{D}$ is a derivation on amenable
$C(X)$. Therefore $\tilde{D} = 0$, and then
   $$D(f\sigma(a))= \tilde{D}(f)\cdot {\sigma}(a) +f\cdot \tilde{D}(\sigma(a)) = 0  \ \ \ (a \in A, f \in C(X)).$$ Whence $D = 0$ on the linear span of
    $ \lbrace f\sigma(a)    : f \in C(X),  a \in A\rbrace $
    which is dense in the linear span of $ \lbrace f a    : f \in C(X),  a \in A \rbrace
    $, by the density of range of $ \sigma$. The latter is itself
    dense in $ C(X, A)$ by [2], so that $D = 0$ on the whole $C(X,
    A)$.
\end{proof}
In Proposition 3.2, as an special case, we may suppose that
$(e_{v})$ is itself a bounded approximate identity for $A$. Indeed
if $(e_{v})$ is a bounded approximate identity for $A$, then the
density of the range of $\sigma$ shows that $(\sigma(e_{v}))$ is
still a bounded approximate identity for $A$.

The following was proved in \cite [Proposition 18]{T. I}, and so we
omit its proof.
\begin{prop}
Suppose that $\sigma \in Hom(A)$ and $\tau \in Hom(B)$, where $A$
and $B$ are Banach algebras such that $A$ is commutative and
$\sigma$ has a dense range. Suppose that $ \phi : A \longrightarrow
B$ is a continuous epimorphism for which $ \tau \phi = \phi \sigma$.
If $A$ is $\sigma$-weakly amenable, then $B$ is $\tau$-weakly
amenable.
\end{prop}
The following is the converse of Proposition 3.2.
\begin{theo}
 Let $X$ be a compact Hausdorff space, let $A$ be a commutative Banach algebra
and let $\tau \in Hom(C(X, A))$ with a dense range such that
$T_{x_0} \tau = \tilde{\tau}_{x_0} T_{x_0}$, for some $x_0 \in X$.
If $C(X, A)$ is $\tau$-weakly amenable, then $A$ is
$\tilde{\tau}_{x_0}$-weakly amenable.
\end{theo}
\begin{proof}
We use Proposition 3.3.
\end{proof}
We extend [1, Theorem 1.8.4] as follows, where the proof reads
somehow the same lines.
\begin{prop} \label{3.3}
Let $I$ be an ideal in a commutative algebra $A$, $E$ be an
$A$-module and $ D: I \longrightarrow E$ be a $\sigma $-derivation.
Then the map $$\tilde{D} : I \times A \longrightarrow E , \ \ \
\tilde{D}(a,b) := D(ab)-\sigma(b) \cdot D(a)$$ is  bilinear such
that

(i)  $\tilde{D}(a, b) = \sigma(a)\cdot D(b) \ \ \        (a, b \in
I)$;

(ii) for each $a\in I^{2}$, the map $b \longmapsto \tilde{D}(a,b),$
$ A\longrightarrow E$ is a $\sigma$-derivation.
\end{prop}
\begin{proof}
We only prove the clause (ii). For $a_{1}, a_{2} \in I$, and $b_{1},
b_{2} \in A$, we have
 \begin{align*}
  \tilde{D}(a_{1}a_{2}, b_{1}b_{2})
  &=D(a_{1}a_{2}b_{1}b_{2})
  -\sigma (b_{1}b_{2}) \cdot D(a_{1}a_{2}) \\
  &=\sigma (a_{1}b_{1})\cdot D(a_{2}b_{2})+ \sigma(a_{2}b_{2}) \cdot D(a_{1}b_{1})
  -\sigma(b_{1}b_{2}) \cdot D(a_{1}a_{2}) \\
  &= \sigma(b_{1}) \cdot [\sigma(a_{1}) \cdot D(a_{2}b_{2})- \sigma(b_{2}) \sigma (a_{1})\cdot D(a_{2})] \\
  &+ \sigma (b_{2}) \cdot  [\sigma(b_{1}a_{1})\cdot
   D(a_{2}) +\sigma(a_{2}) \cdot D(a_{1}b_{1})
  - \sigma(b_{1}) \cdot D(a_{1}a_{2})] \\
  &=\sigma(b_{1}) \cdot [D(a_{1}a_{2}b_{2})-\sigma(a_{2}b_{2}) \cdot D(a_{1})
  -\sigma(b_{2}a_{1}) \cdot D(a_{2})]  \\
&+ \sigma(b_{2}) \cdot  [D(a_{1}a_{2}b_{1})-\sigma (b_{1}) \cdot D(a_{1}a_{2})] \\
& = \sigma(b_{1}) \cdot [D(a_{1}a_{2}b_{2})-\sigma(b_{2}) \cdot D(a_{1}a_{2})] \\
&+
 \sigma(b_{2}) \cdot [D(a_{1}a_{2}b_{1})-\sigma (b_{1}) \cdot D(a_{1}a_{2})]\\
&=\sigma(b_{1}) \cdot \tilde{D}(a_{1}a_{2},b_{2}) + \sigma(b_{2}) \cdot \tilde{D}(a_{1}a_{2},b_{1})
\end{align*}
as required.
\end{proof}
The following is an analogue of [1, Lemma 2.8.68].
\begin{lemm} \label{3.5}
Let $A$ be a $\sigma$-weakly amenable commutative Banach algebra,
$I$ be a closed ideal in $A$, and $E$ be a
 Banach $A$-module. Take $ \sigma \in Hom(A)$
with a dense range such that $\sigma(I)\subseteq I$. Then $D\mid_{
I^{4}}=0$ for each $D\in Z^{1}_{\sigma}(I,E)$.
\end{lemm}
\begin{proof}
We first observe that $F:={}_{A}{\mathcal{B}(I,E)}$ is a Banach
$A$-module for the action $(a \cdot T)(b) =T(ab)$,
 $(a\in A, b\in I)$. Then we notice that the map $j: E\longrightarrow F$ with $j(x)(a)=a \cdot x$,  $(a\in I, x\in E)$ belongs
  to ${}_{A}{\mathcal{B}(E,F)}$. Whence $j\circ D\in Z^{1}_{\sigma}(I,
  E)$. Clearly the map  $$ \tilde{D}: I \times A \rightarrow F ,  \ \ \ (a, b) \longmapsto (j\circ D)(a, b)
  -\sigma(b) \cdot (j \circ D)(a)            $$ is bilinear. Therefore $\tilde{D}(a, b) = \sigma\ (a) \cdot (j \circ
  D)(b)$, by Proposition 3.5 (i), for $a, b\in I$. Take $a\in I^{2}$. By Proposition 3.5 (ii), the map $ A \longrightarrow F
  $, $ b \longmapsto \tilde{D}(a, b)$, is a $\sigma$-derivation. It follows from Proposition 3.1, that this map is zero and so
$\tilde{D}(I^{2}\times A) =0$. Hence, for $a\in I^{2}$ and $b, c \in
I$, we have

$$ \sigma (a)c \cdot D(b)= (j\circ D)(b)(\sigma(a)c) = (\sigma (a) \cdot (j \circ
D)(b))(c) = \tilde{D}(a, b)(c) = 0. $$ In particular, choosing $c
\in \sigma(I)$, we may see that $ \sigma (I^{3}) \cdot D(I) =0$, and
whence $D\mid I^{4}=0$.
\end{proof}
\begin{prop}
Let $A$ be a $\sigma$-weakly amenable commutative Banach algebra,
 $I$ be a closed ideal in $A$, and  $ \sigma \in Hom(A)$ with
a dense range such that  $\overline{\sigma(I)}=I$. Then $I$ is
$\sigma$-weakly amenable if and only if $\overline{I^{2}} = I$.
\end{prop}
\begin{proof}
If $I$ is $\sigma$-weakly amenable, then $\overline{I^{2}}=I$ by
\cite[Theorem 6]{T. I}.

Conversely if $\overline{I^{2}}=I$, then $\overline{I^{4}}=I$.
Suppose that $D\in Z^{1}_{\sigma}(I , I^{*})$. Then by Lemma
\ref{3.5}, $D\mid_{I^{4}}=0$ so that $D=0$. Hence $I$ is
$\sigma$-weakly amenable.
\end{proof}

\begin{lemm}\label{3.8}
Let $X$ be a compact Hausdorff space, let $A$ be a Banach algebra
and let $\sigma \in Hom(A)$. If $\sigma$ has a dense range, then
$\tilde{\sigma}$ has a dense range as well.
\end{lemm}
\begin{proof}
Take $f \in C(X)$ and $a \in A$ and put $h:= fa \in C(X, A)$. By the
assumption, there exists a sequence $(a_{n}) \subseteq A$ such that
$ \lim_n \sigma(a_{n}) = a$. Define $h_{n}:= fa_{n} $, $( n =
1,2,...)$. Then it is easy to verify that $ \lim_n
\tilde{\sigma}(h_{n}) = h$. This completes the proof, since the set
of all linear combinations of elements of $C(X, A)$ of the form $fa$
$(f \in C(X), a \in A)$ is dense in $C(X, A)$ [2].
\end{proof}
We recall that a homomorphism $\sigma \in Hom(A)$ is extended to a
homomorphism $\sigma^{\sharp} \in Hom (A^{\sharp})$ through
$\sigma^{\sharp}(a + \lambda e) = \sigma(a) + \lambda e $ $(a \in A,
\lambda \in \mathbb{C})$, where $e$ is the identity of $A^{\sharp}$,
the unitization of $A$.

Now, we are ready to prove our last goal.
\begin{theo}
Let $X$ be a compact Hausdorff space, let $A$ be a commutative
Banach algebra and let $\sigma \in Hom(A)$ with a dense range. If
$A$ is ${\sigma}$-weakly amenable, then $ C(X,A)$ is
$\tilde{\sigma}$-weakly amenable.
\end{theo}
\begin{proof}
Suppose that $A$ is ${\sigma}$-weakly amenable. By [6, Theorem 12],
$A^{\sharp}$ is $\sigma^{\sharp}$-weakly amenable. Applying
Proposition 3.2, we see that $C(X, A^{\sharp})$ is
$\tilde{\sigma^{\sharp}}$-weakly amenable. Our assumptions together
with [6, Theorem 6], imply that $A^{2}$ is dense in $A$. We learn
from the proof of [8, Theorem 1] that ${C(X, A)^{2}}$ is dense in
$C(X, A)$, and also $C(X, A)$ is a closed ideal of $C(X,
A^{\sharp})$. Next, it is easy to verify that $
\tilde{\sigma^{\sharp}}(f) = \tilde{\sigma}(f)$, for all $f \in C(X,
A)$. Hence $$ \overline{\tilde{\sigma^{\sharp}}(C(X, A))} =
\overline{\tilde{\sigma}(C(X, A))} = C(X, A) $$
 by Lemma \ref {3.8}. Now, an application of Proposition 3.7 yields
 that $C(X, A)$ is $\tilde{\sigma}$-weakly amenable.
\end{proof}


 \end{document}